\documentclass[12pt]{amsart}
\usepackage{geometry}     
\usepackage{graphicx}
\usepackage{color}

\usepackage{amsthm}
\usepackage{amsmath}
\usepackage{amsfonts}
\usepackage{enumerate}
\usepackage{amssymb}
\usepackage{epstopdf}
\usepackage{hyperref}

\newtheorem{theorem}{Theorem}[section]
\newtheorem{proposition}[theorem]{Proposition}
\newtheorem{lemma}[theorem]{Lemma}

\newcommand{\R}{\mathbb{R}}
\newcommand{\RP}{\mathbb{RP}}
\newcommand{\Z}{\mathbb{Z}}

\theoremstyle{definition}
\newtheorem{definition}[theorem]{Definition}

\title{Families of short cycles on Riemannian surfaces}
\author{Yevgeny Liokumovich}

\begin{document}

\maketitle

\begin{abstract}
Let $M$ be a closed Riemannian surface
of genus $g$.
We construct a family 
of 1-cycles on $M$ that represents a 
non-trivial element of the k'th homology group of 
the space of cycles and such that the mass of each cycle is bounded above by
$C \max\{\sqrt{k}, \sqrt{g}\} \sqrt{Area(M)}$.
This result is optimal up to 
a multiplicative constant.
\end{abstract}

\section{Introduction}
\label{sec:Introduction}

Let $M$ be a closed Riemannian 2-dimensional manifold and
let $Z_1(M,\Z_2)$ denote 
the space of mod 2 flat 1-cycles in $M$.
Let $Z_1^0$ denote the connected component of $Z_1(M,\Z_2)$
consisting of all null-homologous cycles in $M$.
It follows from the work of Almgren \cite{Almgren1962}
that $Z_1^0$ is weakly homotopy equivalent
to the Eilenberg-MacLane space $K(\Z_2,1) \simeq \RP^{\infty}$.
We say that a family of cycles $f: \RP^k \rightarrow Z_1^0$
is a $k$-sweepout if it represents the non-zero element
of the $k$'th homology group $H_k(Z_1^0,\Z_2) \cong \Z_2$.

Here is the main result of this paper.

\begin{theorem} \label{main}
Let $M$ be a 2-dimensional closed 
Riemannian manifold of genus $g$.
For each $k$ there exists a $k$-sweepout
$\mathcal{Z}_k = \{z_t\}_{t\in \RP^k}$ of $M$,
such that for each $t \in \RP^k$ the mass of $z_t$ 
is bounded above by
$1600 \max\{\sqrt{k}, \sqrt{g}\} \sqrt{Area(M)}$.
\end{theorem}

$k$-sweepouts have been studied by Gromov in \cite{Gromov1988},
\cite{Gromov2003} and \cite{Gromov2009} and
by Guth in \cite{Guth2008}.
More recently, in \cite{MarquesNeves2013} Marques and Neves used 
$k$-sweepouts to prove existence of infinitely many
minimal hypersurfaces in manifolds of positive Ricci curvature.
In \cite{Glynn-AdeyLiokumovich} Glynn-Adey and the author 
obtained upper bounds for volumes of these hypersurfaces.
 
In the case of surfaces Balacheff and Sabourau \cite{BalacheffSabourau2010}
constructed a sweepout of $M$ by 1-cycles of mass bounded
by $C \sqrt{(g+1)Area(M)}$.
This corresponds to the case $k=1$ of Theorem \ref{main}.
Different proofs of their result, improving the value of an upper bound
for the constant $C$, were given in \cite{Liokumovich2013},
\cite{Glynn-AdeyLiokumovich}.
The proof of Balacheff and Sabourau relies 
on the estimate of Li and Yau \cite{LiYau1982}
for the first eigenvalue of the Laplacian.
In this paper we give an elementary construction
of $k$-sweepouts using only the thin-thick decomposition of
hyperbolic surfaces and the length-area method. 
 
The upper bound in Theorem \ref{main} is optimal up 
to a constant. Brooks constructed examples of 
closed hyperbolic surfaces of arbitrarily large genus
such that any $1$-sweepout of $\Sigma_g$ must contain a cycle
of mass greater than $c \sqrt{g}$ for some 
$c>0$. On the other hand, Gromov showed in \cite{Gromov1988}
that a $k$-sweepout of the round $n$-sphere by $(n-1)$-cycles 
must contain a cycle of mass greater than $c k^{\frac{1}{n}}$
for a constant $c>0$.
To prove this Gromov observed that if $\{U_i\}$ is
a collection of $k$ disjoint measurable subsets in 
 $M$ and $z_t$ is a $k$-sweepout, then there will be
 a cycle $z_t$ that separates each $U_i$ into two subsets of 
 equal area.
 Gromov's arguments were generalized
 and extended by Guth in \cite{Guth2009}.
 In that paper Guth proves nearly optimal lower and upper bounds
 for all homology classes of the space of mod 2 $m$-dimensional
 cycles on the $n$-dimensional round sphere.

In \cite{Gromov1988} Gromov suggested that finding bounds on the maximal 
mass of a cycle in an optimal $k$-sweepout 
can be thought of 
as a non-linear analogue of the spectral problem on $M$.
Arguments in our paper, especially the use of
the length-area method, were inspired by and are similar to
the estimates for the eigenvalues of the Laplace operator
on Riemannian manifolds in the works of Hersch \cite{Hersch},
Yau \cite{Yau1975}, Yang and Yau \cite{YangYau},
Korevaar \cite{Korevaar1993}, Gromov \cite{Gromov1993}, 
Grigoryan, Netrusov and Yau \cite{Grigoryan2004}, Colbois and Maertens
\cite{ColboisMaerten}, and Hassannezhad \cite{Hassanezhad}.

\vspace{0.1in}

\textbf{Acknowledgements.} 
I am grateful to Misha Gromov for explaining the 
connection between $k$-sweepouts and spectral problems
and for suggesting methods of Hersch \cite{Hersch} and Korevaar \cite{Korevaar1993}
for the kind of problems considered in this paper.
I would like to thank my advisers Alexander Nabutovsky and Regina Rotman
for many very valuable discussions and for important comments on 
the first draft of this paper.
I am grateful to anonymous referees for careful reading of the article and excellent suggestions that
helped to improve the exposition.

The author was partially supported by the Queen Elizabeth II/ Israel Halperin Graduate Scholarship.

\section{Outline of the proof} \label{outline}
Let $M$ be a closed surface of area $1$. 
Suppose we can cover $M$ by $k$ sets $U_i$ with piecewise smooth boundary
and disjoint interiors,
each of area $\sim \frac{1}{k}$, and such that the boundary length
of each set is $\sim \frac{1}{\sqrt{k}}$.
Assume furthermore that for each $U_i$ there exists
a $1$-sweepout of $U_i$ by cycles of length at most $\sim \frac{1}{\sqrt{k}}$.
We can now sweep out all of $M$ as follows.
First we sweep out $U_1$, starting on a $0$-cycle and
ending on the boundary of $U_1$. We hold cycle
$\partial U_1$ fixed and start adding to it
a sweepout of $U_2$ and so on.
Eventually cycles in the boundaries of $U_i$'s
will overlap and cancel out.

Denote this sweepout of $M$ by $z_t$ and consider
a cycle $z = \sum_{i=1}^k z_{t_i}$,
where $\{ t_i \}$ are $k$ different moments of time.
Each $z_{t_i}$ can be decomposed into two parts:
one that lies in $\bigcup \partial U_i$
and one that is contained in only one of the sets $U_i$
and has mass at most $\sim \frac{1}{\sqrt{k}}$.
Since the cycles are mod 2, the parts that overlap 
in $\bigcup \partial U_i$ will cancel out,
so $mass(z) \lesssim \sqrt{k}$.
There exists a $k$-sweepout
of $M$ that consists of cycles like $z$ and therefore
satisfies the desired upper bound.

The idea described above was successfully used by Gromov and Guth
to bound volumes of $k$-sweepouts in various contexts.

If $M$ is a Riemannian 2-sphere then one can find 
a covering of $M$ by $k$ sets as described above.
This can be done using the length-area method as described in 
Section \ref{sec: good regions}.
To construct a $1$-sweepout of $U_i$'s we use the following
idea from the work of Balacheff and Sabourau \cite{BalacheffSabourau2010}.
First, we find a relative $1$-cycle $c_1$
subdividing $U_i$ into two sets $U_i^1$ and $U_i^2$ each of area
$\leq r Area(U_i)$ for some fixed $r \in (0,1/2)$, such that the length of $c_1$
is bounded above by $\sim \sqrt{Area(U_i)}$. 
Let $W_1(U)$ denote the maximal length of a relative cycle in 
an ``optimal" sweepout of $U$ (precise definition will be given
in section \ref{preliminaries}).
Given a sweepout
of each of $U_i^1$ and $U_i^2$ by relative cycles we can assemble them
into a sweepout of $U_i$ by attaching pieces of $c_1$ to some of these
cycles. It follows then that $W_1(U_i)$ is bounded above
by $\sim \max \{ W(U_i^1), W(U_i^2)\} + \sqrt{Area(U_i)}$.
We can repeat this process and subdivide $U_i^j$ into
two subsets $U_i^{j,1}$ and $U_i^{j,2}$.
After $n$ iterations we obtain
 $W_1(U_i) \lesssim \max \{U_i^{j_1, ..., j_n} \} + 
\sum_{i=0}^{n-1} r^{i} \sqrt{Area(U_i)}$ and the
areas of sets $U_i^{j_1, ..., j_n}$ are at most
$r^n Area(U_i)$. Since the geometric series $\sum_{i=0}^{n-1} r^{i}$
converges as $n \rightarrow \infty$ the above argument reduces the problem
of bounding the $1$-width of $U_i$ to a problem of bounding
the $1$-width of a subset $U_i^{j_1, ..., j_n} \subset U_i$ of arbitrarily small area. To accomplish this we cut $U_i^{j_1, ..., j_n}$
into pieces which are $(1+\epsilon)$-bilipschitz
to open subsets of Euclidean plane and apply an 
argument of Guth \cite{Guth2007}.

However, if the surface has genus greater than $k$ the above
argument may not work. It may happen that
every collection of $k$ open sets of approximately equal areas 
that cover $M$ have large length of the boundary and some of these open sets do not admit a sweepout by short cycles. This happens, for example, for hyperbolic surfaces
of high genus constructed by Brooks \cite{Brooks1986}.

Instead we will first cover $M$ by $\sim g$ `good regions' 
$V_i$ (where $g$ is the genus). These regions can have arbitrary areas,
but they have the following nice properties:
\begin{enumerate}
\item There exists a sweepout of $V_i$
by relative 1-cycles of length at most $\sim \sqrt{Area(V_i)} $ 
\item We can subdivide 
$V_i$ into $m$ (where $m$ is any positive integer)
subsets of approximately equal areas, such that
the length of the union of their boundaries
is at most $\sim \sqrt{m} \sqrt{Area(V_i)} +l(\partial V_i)$
\end{enumerate}
So for our purposes these good regions are as good as
subsets of the sphere. 
We will then subdivide them into subsets of the right area.
The value of $m$ that we choose for each region
$V_i$ will depend on $k$ and the area of $V_i$.

To obtain these good regions
we use uniformization theorem and the length-area method.
By uniformization theorem a surface of genus $g \geq 2$
is conformally equivalent to a hyperbolic surface.
P. Buser used thin-thick decomposition 
to construct a tessellation of a hyperbolic surface
by polygons of approximately equal areas with some special properties.
The thin part of the surface in this tessellation is covered by long and narrow rectangles
and the thick part is covered by triangles that are close to equilateral 
triangles. For us the most important thing about this tessellations
is that every polygon contains at most $c$ other polygons
in its $1/2$-neighbourhood. Our good regions
will be those that are covered by at most $c$
polygons from this tessellation. 

To control lengths of the boundaries of 
good regions we observe that if 
a family of concentric geodesic circles (i.e. level sets of the distance function)
on the hyperbolic surface
(conformal to our surface $M$) covers a set of small area,
when measured with the original (non-hyperbolic) metric,
then some of these circles must be short in the original metric.
This is a classical observation sometimes called the length-area method
(see Section \ref{sec: length-area}). We use it to find 
short cycles on $M$ in $1/2-$neighbourhood of a polygon
from the hyperbolic tessellation.
Actually, the length of the boundary of each individual
good region in our construction may be comparatively long, 
but the total length
of the union of their boundaries will be at most
$\sim \max\{\sqrt{g},\sqrt{k}\}$. Moreover, after
we subdivide each good region into smaller parts 
using property (2) above so that area
of each part is at most $\sim \frac{1}{k}$, the total 
length of the union of the boundaries of all parts
will still be at most $\sim \max\{\sqrt{g},\sqrt{k}\}$.
This is sufficient to bound lengths of
$k$-sweepouts using the argument described above.

Here's the plan of the paper. In Section \ref{preliminaries}
we define $k$-sweepouts and a technical notion of monotone
sweepouts. These sweepouts have a nice property that it is
easy to glue 
two short monotone sweepouts of adjacent regions into a
short monotone sweepout of their union.
In Section \ref{sec: length-area} we use the length-area method to prove a key lemma for finding
subsets of $M$ with small length of the boundary.
In Section \ref{tessellation} we describe Buser's tessellation 
$\mathcal{T}$ of a hyperbolic surface by quadrilaterals
and triangles.
In Section \ref{sec: Guth} we describe Guth's construction
of sweepouts of open subsets of $\R^2$.
We use this result as the base of induction
in the proof that a subset of $M$ of very small area
admits a sweepout by short cycles.
In Section \ref{sec: good regions} we prove that if a
subset $U$ of $M$ can be covered by at most $40$
elements of $\mathcal{T}$ then it admits a sweepout
by cycles of length at most $\sim \sqrt{Area(U)}$.
In Section \ref{sec: covering} we construct a covering of $M$
by sets that are contained in at most $40$ elements of $\mathcal{T}$
and have area at most $\frac{Area(M)}{k}$ and finish the proof of the theorem.

\section{Preliminaries} \label{preliminaries}

For the definition of the space of mod 2 cycles with flat metric
we refer the reader to \cite{Federer1969} or a concise description
in \cite[Section~2]{BalacheffSabourau2010}, which will be sufficient 
for our purposes. 

In \cite{Almgren1962} Almgren constructed maps  
from homotopy groups of the integral
cycle space $\pi_k(Z_{m}(M^n, \Z);0)$
to homology groups of the manifold $H_{k+m}(M^n,\Z)$
and proved that these maps are isomorphisms for all 
non-negative integers $k$ and $m$.
Almgren's proof works for $\Z_2$ coefficients as well.
For a surface $M$ we have an isomorphism 
$\pi_k(Z_{1}(M, \Z_2);0) \cong H_{k+1}(M,\Z_2)$.
Since homology groups of $M$ are zero for $k>1$, the connected component $Z_1^0$ of $Z_{1}(M, \Z_2)$,
$0 \in Z_1^0$, is weakly homotopy equivalent to 
the Eilenberg-MacLane space $K(\Z_2,1)\simeq \RP^{\infty}$.

For a surface $M$ Almgren's map 
$F_{A}:\pi_1(Z_{1}(M, \Z_2),0) \rightarrow H_2(M,\Z_2)$ 
is defined as follows.
Consider a loop $z_t: S^1 \rightarrow Z_{1}(M, \Z_2)$ representing
some class of the fundamental group and pick a fine subdivision 
$\{t_1,..., t_n\}$  of $S^1$. For each $t_i$ cycle $z_{t_i}$ can 
be approximated by a cycle that consists of a 
finite collection of Lipschitz circles. 
If $c_i$ and $c_{i+1}$ are two such approximations of
$z_{t_i}$ and $z_{t_{i+1}}$ respectively,
we can find an area minimizing chain $A_i$
with $\partial A_i = c_i - c_{i+1}$
We can then assemble chains $A_i$ into a 2-cycle 
that represents an element of $H_2(M,\Z_2)$.
It turns out that if the subdivision and approximations
are fine enough then the 2-cycle will represent the same 
element in the homology independent
of the particular subdivision and approximations.

We say that $\{z_t\}_{t \in \RP^1}$ is a sweepout (or 1-sweepout) of $M$ if
loop $\{z_t\}$ is non-contractible
in $Z_{1}(M, \Z_2)$, i.e. $F_A([z_t])\neq 0$.
More generally, we say that $\{z_t\}_{t \in \RP^k}$
is a $k$-sweepout if 
it represents the non-zero element of $H_k(Z_1^0)\cong \Z_2$.
The ring structure of $H^*(Z_1^0,\Z_2) \cong \Z_2[a]$,
where $a$ is the non-zero class of $H^1(Z_1^0,\Z_2)$,
provides a useful criterion for when a family is a $k$-sweepout.
We have that map $f: \RP^k \rightarrow Z_1^0$
is a $k$-sweepout if and only if the pull-back $f^*(a^k) \neq 0$.


We will frequently need to consider sweepouts
of manifolds with boundary.
In this case we consider the space of cycles
relative to the boundary and all definitions above
carry over to this setting.

The 1-sweepouts that we construct in this paper are 
nicer than an arbitrary 1-sweepout. After a small perturbation
different cycles in 
it will not intersect each other and one can turn them
into level sets of a function $f:M \rightarrow \R$.
We summarize this in the following definition. 

\begin{definition} \label{monotone}
Let $M$ be a Riemannian surface (possibly with boundary).
Let $int(M)$ denote the interior of $M$.
We say that $z_t$ is a monotone sweepout if $z_t$
is a sweepout of $M$ and for each $t$ cycle
$z_t$ can be represented by a finite collection
of points and piecewise smooth simple closed curves,
which satisfy the following condition. 
There exists a family of nested subsets $A_t \subset M$,
$A_{t'} \subset A_t$ for all $t'<t$, such that $z_t$ contains
$\partial A_t \setminus \partial M$ and is contained in $\partial A_t$.
\end{definition}

Since the cycles are nested and they can be glued into
the fundamental class of $M$, it follows that
$A_0$ is collection of points and $A_1$ is all of $M$.
Below we use this property to concatenate
sweepouts of two adjacent regions.

\begin{lemma} \label{concat}

Let $M$ be a Riemannian surface, possibly with boundary, and 
let $\gamma$ be a relative 1-cycle composed of finitely many
piecewise smooth closed curves that have not self-intersections
or pairwise intersections and
separate $M$ into $M_1$ and $M_2$. Suppose there exist
monotone sweepouts of $M_1$ and $M_2$ of length
at most $L$. Then there exists a monotone sweepout
of $M$ by cycles $z_t$, such that we can decompose
$z_t$ as a sum of $1$-chains $z^1 _t + z^2 _t$,
where $l(z^1_t) \leq L +\epsilon$ and $z^2 _t$ is 
contained in $\gamma$.
\end{lemma}

\begin{proof}
By definition of a monotone sweepout for each $i = 1,2$ there exists a
family $A^i_t$ of nested sets with 
$int(M_i) \cap \partial A^i _t \subseteq z^i_t \subseteq  \partial A^i _t$.
After a small perturbation that keeps $A^i _t$'s
nested and increases
lengths of cycles by at most $\epsilon$ we can assume that 
$\partial A^i _t$ will intersect $\gamma$ in a (possibly empty) finite collection of
arcs and closed curves $I^i_t$ with $I^i_t \subseteq I^i_{t'}$ if $t \leq t'$.

Define $A_{t} = A^1 _{2t}$ for $t\in [0,\frac{1}{2}]$
and $A_t = A^1_1 \cup A^2_{\frac{t+1}{2}}$ for $t \in (\frac{1}{2},1]$.
We define sweepout $z_t= \overline{\partial A_t \cap int(M)}$.
For $t \leq \frac{1}{2}$ each cycle $z_t$ can be decomposed into 
a chain that is contained in $z^1_{2t}$ and a 
chain $I^1_t\subset \gamma$. 
For $t > \frac{1}{2}$  cycle $z_t$ can be decomposed into 
a chain that is contained in $z^2_{\frac{t+1}{2}}$ and a 
chain $\gamma \setminus I^2_t$.
\end{proof}

\section{Length-area method} \label{sec: length-area}

Given a closed Riemannian surface $(M,h)$ 
by uniformization theorem there exists a conformal
diffeomorphism $\phi: (M,h) \rightarrow (M,h_0)$
from $(M,h)$ to a surface of constant curvature $(M,h_0)$. 
This conformal equivalence will play a key role
in our construction of parametric sweepouts.
For a subset $U \subset M$ we will write $\mu_0(U)$
to denote its area with respect to metric $h_0$
and  $\mu(U)$ to denote its area with
respect to $h$.
Similarly,
we will write $d(x,y)$, $B(x,r)$ and 
$\nabla$  to denote distance 
	 function, closed metric ball of 
	 radius $r$ about $x$, and gradient with respect 
	 to $h$ and we let $d^0(x,y)$,
	 $B^0(x,r)$, $\nabla^0$ denote the corresponding 
	 quantities with respect to $h_0$.

A key tool in this paper is an old technique sometimes
called the length-area method 
(see, for example, \cite{jenkins1958}). 
It is based on the observation that
the $n'$th power of the absolute value of the gradient of a function
(where $n$ is the dimension of the space)
times the volume element is a conformal invariant.
Using this observation and coarea formula we can obtain the following
lemma, which will be used throughout the paper.

Let $N_{r} ^0(U)$ denote the set $\{x \in M: d^0(x,U) < r \}$
and $A_{r}^0(U) = N_{r} ^0(U) \setminus U$.

\begin{lemma} \label{length-area}
Let $U$ and $V$ be open subsets of $M$ with
$U \subset V \subset M$. For any $r>0$ there exists an
open set $U'$ with $U \subset U' \subset V \cap N_{r}^0(U)$, such that
$l( \partial U' \cap V) \leq 
\frac{\sqrt{\mu_0(A_{r} ^0(U) \cap V)}}{r} \sqrt{\mu(A_{r} ^0(U) \cap V)} $.
\end{lemma}

\begin{proof}
Let $d^0_V$ denote the distance function induced by the restriction of
Riemannian metric $h_0$ to the open set $V$. Observe that
for any two points $x$ and $y$ in $V$ we have $d^0(x,y) \leq d^0_V(x,y)$.
In particular, we have that 
$A_{r}^0(U,V) = \{x \in V: d^0_V(x,U) < r \} \setminus U 
\subset A_{r}^0(U) \cap V$.
Define a function $f: V \setminus U \rightarrow \mathbb{R}$
by setting $f(x)= d^0_V(x,U)$. 
By Rademacher's theorem $f$ is differentiable almost everywhere.
By coarea formula we have
$$ \int_{t=0} ^{r} l(f^{-1}(t)) dt =\int_{A_{r}^0(U,V)} |\nabla f| d \mu$$
By Cauchy-Schwartz inequality this quantity can be bounded above by 
$$(\int_{A_{r}^0(U,V)} |\nabla f|^{2} d\mu)^{1/2} \mu (A_{r}^0(U,V))^{1/2}$$
We observe that $|\nabla f|^{2} dV$ is a conformal invariant, so
$$\int_{A_{r}^0(U,V)} |\nabla f|^{2} d\mu
=\int_{A_{r}^0(U,V)} |\nabla^0 f|^{2} d\mu_0 = \mu_0 (A_{r}^0(U,V))$$
It follows that for some $l \in [0,r]$ the set $U'= f^{-1}([0,l]) \cup U$
 has boundary length at most 
 $\frac{\sqrt{\mu_0(A_{r} ^0(U,V) )}}{r} \sqrt{\mu(A_{r} ^0(U,V))}$.

\end{proof}

\section{Tessellations of hyperbolic surfaces} \label{tessellation}

We use the following tessellation of a Riemann surface
due to Buser.

\begin{proposition} \label{hyp covering}
(Buser)
Let $\Sigma$ be a closed hyperbolic surface.
There exists a tessellation of $\Sigma$
into polygons $\mathcal{T}=\mathcal{T}_1 \cup \mathcal{T}_2$
with the following properties:

1. $\mathcal{T}_1$ is a collection of triangles
with sidelengths
between $\log(2)$ and $2\log(2)$
and areas between $0.19$ and $0.55$.

2. $\mathcal{T}_2$ is a collection of quadrilaterals 
(see figure \ref{trirectangle} ) with 
three right angles and one angle $\phi> \pi/3$.
The sidelengths satisfy the following relations:
$a \leq \log(2)/2$, $\log(2)/2 \leq c \leq 0.45$
and $b \geq d \geq 0.57$. The area of each
quadrilateral is between $0.26$ and $0.34$.

3. For each polygon $T \in \mathcal{T}$ the $1/2$-neighbourhood
of $T$ is contained in at most $40$ polygons of $\mathcal{T}$.
\end{proposition}

\begin{figure} 
   \centering	
	\includegraphics[scale=0.3]{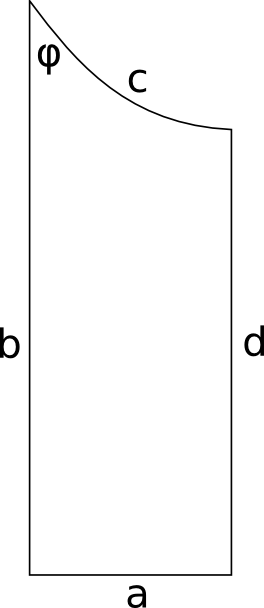}
	\caption{$\mathcal{T}_2$ consists of hyperbolic quadrilaterals with three right angles.}
	\label{trirectangle}
\end{figure}

\begin{proof}
The construction of Buser (\cite{buser}, p.116-121)
relies on the thin-thick decomposition of $\Sigma$.
Let $\beta_1, ..., \beta_k$ be the
set of all simple closed geodesics of length
$\leq  \log(2)$ and let 
$w_i=arcsinh(\frac{1}{sinh(\frac{1}{2}|\beta_i|)})>1 $.
Then the tubular neighbourhood of $\beta_i$
$C_i = \{p \in \Sigma| d(p, \beta_i) \leq w_i \}$
is isometric to the cylinder
$[-w_i,w_i] \times S^1$ with the Riemannian metric
$ds^2 = d \rho^2 + |\beta_i|^2 cosh^2(\rho) dt^2$.
Moreover, the cylinders $C_i$ are disjoint.

In each collar $C_i$
Buser defines two isometric annular regions, which
he calls trigons. One boundary component of the trigon
is the closed geodesic $\beta_i$ and the other boundary
component consists of two geodesic arcs of equal length.
The endpoints of these geodesic arcs lie 
at a distance $w_i - \log(2)/2$ from $\beta_i$.
Each trigon can be subdivided into four isometric
quadrilateral as on Figure $\ref{trirectangle}$.
These quadrilaterals have three right angle.
A computation then yields the desired bounds
on the sidelengths and the fourth angle.
We define $\mathcal{T}_2$ to be the collection of all
such quadrilaterals (eight in each collar).

In the remaining (thick) part of
$\Sigma$ the injectivity radius at a point $x$ is
bounded from below by $\min \{\log(2),d(x,V_2) \}$,
where $V$ denotes the set of vertices of quadrilaterals 
in $\mathcal{T}_2$. Buser considers a maximal set of points
at pairwise distances at least $\log(2)$. He then defines 
a geodesic triangulation of the thick part
with this set as the set of vertices.

To prove the last assertion we observe that the worst case is when 
$T$ is a triangle that is not adjacent to any of
the quadrilaterals. As computed by Buser,
all angles of the triangle are bounded
below by $22.6^{\circ}$. It follows that $1/2$-neighbourhood
of $T$ can be covered by less than $40$ triangles.
\end{proof}


\section{Sweepouts of open subsets of $\R^2$} \label{sec: Guth}

Our proof of Proposition \ref{small covering sweepout}
relies on its Euclidean analogue.
Namely, we need to know that for any open subset $U$ of Euclidian
plane there exists a sweepout of $U$ by relative
cycles of small length.
This result was proved by Guth in \cite{Guth2007}
along with its high dimensional generalizations.

\begin{theorem}(Guth) \label{Guth}
Let $U \subset \R^2$ be a bounded open subset 
with piecewise smooth boundary. 
There exists a monotone sweepout of $U$
by cycles of length $\leq 3  \sqrt{Area(U)}$.
\end{theorem}

\begin{proof}
We give an outline of the argument in \cite{Guth2007}. 
The 2-dimensional case is significantly easier
than the general inequality obtained by Guth
for $k$-dimensional cycles sweeping out an open subset
in $\R^n$.

At first one may hope that for some line $l \in \R^2$
the projection of $U$ on $l$ will have short fibers.
However, there exist sets in $\mathbb{R}^2$
(known as Besicovitch sets) of arbitrarily small area
such that any such projection will contain a fiber 
 of length larger than $1$.

Instead of sweeping out $U$ by parallel lines
we will use cycles that are mostly contained in 
the 1-skeleton of a square grid.
Scale $U$ to have area $1$. 
If we consider translates of the unit grid
the total length of the intersection 
of the 1-skeleton (i.e. 
the union of the edges) with set $U$ will have,
on average, length equal to $2$.
(This can be seen as follows. 
First we translate the unit grid horizontally
until the intersection of $U$ with vertical lines of the grid
has length $1$; then we translate the grid vertically
until the intersection of $U$ with horizontal
lines of the grid has length $1$ giving us total length $2$).
Consider a large square $C_0=[-N,N]^2$ that contains $U$
and let $l_0 = \partial C_0$.
Let $C_1 = C_0 \setminus [-N,-N+1] \times [N-1,N]$.
Continue removing unit squares one by one (see Figure \ref{euclidean}).
This way we obtain $N^2$ connected unions
of unit squares $C_i$ with boundary in the 1-skeleton of the
unit grid.
Observe that one can homotop $\partial C_i$ to
$\partial C_{i+1}$ via cycles that are contained in 1-skeleton except for a piece of length $1$.

\begin{figure} 
   \centering	
	\includegraphics[scale=0.6]{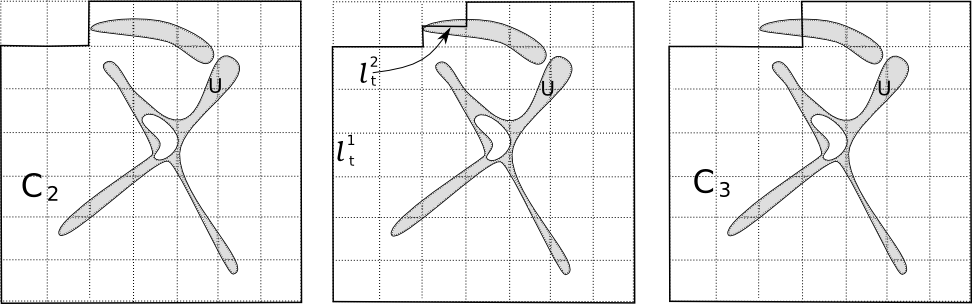}
	\caption{Monotone sweepout of a subset of $\R^2$.} \label{euclidean}
\end{figure}

This gives rise to a family of nested open sets $A_t$,
$A_{\frac{k}{4N^2}}=C_k$, and a homotopy $l_t = \partial A_t = l_t^1 + l_t ^2$,
where $l_t^1$ is contained in the unit grid and 
$l_t ^2$ is either empty or an interval of length $1$.
Defining $z_t= \overline{\partial A_t \cap int(U)}$ we obtain
a monotone sweepout with the desired length bound.


\end{proof}

\section{Sweepouts of subsets covered by a small number of polygons} \label{sec: good regions} 

When the genus $g$ of $M$  is greater than or equal to $2$ we scale
$(M,h_0)$ to have constant curvature $-1$.
By Gauss-Bonnet its volume satisfies $\mu_0(M) = 4 \pi(g-1)$.
By Proposition \ref{hyp covering} there exists a tessellation $\mathcal{T}$
of $M$ into polygons.

When $g$ is equal to $0$ or $1$ we scale the constant curvature space
(sphere, projective plane, torus or a Klein bottle) so that
it has volume $1$. In this case we set $\mathcal{T}$ to consist
of only one element, the whole space $M$.

\begin{lemma} \label{T}
$\mathcal{T}$ satisfies the following properties:

\begin{enumerate}
\item $\#\mathcal{T} \leq \max\{67(g-1),1\}$

\item 
Suppose $\{T_i \}_{i=1}^k \subset \mathcal{T}$, $k\leq 40$,
and let $B^0(x,r)$ be any ball and let $A$ denote the
annulus $B^0(x,\frac{3r}{2}) \setminus B^0(x,r)$.
There exists $42$ balls
$\{B^0(x_j,r)\}$, such that
$A \cap \bigcup T_i \subset \bigcup B^0(x_j,r)$.
\end{enumerate}
\end{lemma}

\begin{proof}
When genus $g \leq 1$ we have $\# \mathcal{T}=1$. 
It is easy to show that an annulus in the plane
 $B(3/2) \setminus B(1) \subset \R^2$
can be covered by $5$ discs of radius $1$. 
A similar covering also works on the round sphere $S^2$.
We conclude that both properties hold when $g \leq 1$.

Suppose $g\geq 1$.
The first property follows since areas of
polygons in $\mathcal{T}$ are bounded from below by 0.19.

To prove the second property we 
consider two cases.
Suppose $B(x,r)$ is a ball with $r \geq 2$.
We can cover every triangle in $\mathcal{T}$ by a ball
of radius $\log(2)<r$. The remaining points of
$A \cap \bigcup T_i$ lie in quadrilaterals.
A quadrilateral $T \in \mathcal{T}$ can be arbitrarily long,
but it has to be narrow: by construction the distance from a point $x$ on one of its long sides to the other long side is at most $0.45$. We can assume that the length of the side $d$ of $T$
(see Fig. \ref{trirectangle})
is greater than $3$ for otherwise we would have that
$T$ is contained in some ball of radius $r$.

Recall from Buser's construction of quadrilaterals that we described 
in the proof of Proposition \ref{hyp covering}
that $T$ is contained in a hyperbolic collar
along with other $7$ isometric quadrilaterals. 
Four of them lie to one side of a closed geodesic
$\beta$ that cuts the collar in the middle and
four of them lie to the other side of $\beta$. Let 
$C_T$ denote the union of the $4$ quadrilaterals
that lie on the same side of $\beta$ as $T$.
We consider two possibilities. Suppose first that
the center of the ball $x$ does not lie in $C_T$.
We observe that in this case $A \cap T$
is contained in a quadrilateral inside $T$
that can be covered by one ball of radius $r$.
Suppose $x \in C_T$. Then $A \cap T$ is contained
in two subsets of $T$ each of which can be covered
by a ball of radius $r$. If other $3$ quadrilaterals
in $C_T$ are not elements of $\mathcal{T}$ it follows that
we need at most $41$ ball to cover $A \cap \bigcup T_i$.
Notice also that all of $A \cap C_T$ can be covered by 
at most $4$ balls of radius $r$.
It follows then that the worst case is when
exactly two quadrilaterals in $C_T$
are elements of $\mathcal{T}$. Then we will need
at most $42$ balls.


\begin{figure}
   \centering	
	\includegraphics[scale=0.8]{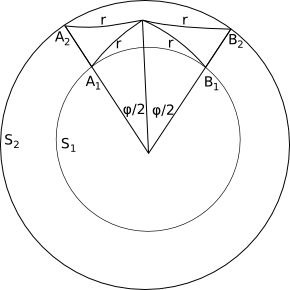}
	\caption{Covering annulus in hyperbolic plane}  \label{covering}
\end{figure}

Suppose $r \leq 2$. In this case we need only $21$
balls $B^0(x_j,r)$ to cover $A$.
This is illustrated on Figure \ref{covering}.
Consider two concentric circles $S_1$ and $S_2$ in the hyperbolic plane
of radii
$r$ and $\frac{3}{2}r$ respectively.
Suppose two geodesic rays emanating from 
$x$ intersect circles $S_1$
and $S_2$ at $A_1$, $B_1$ and $A_2$, $B_2$.
For a correct value of the angle $\phi(r)$
between two geodesic rays we will have
all four points lying on a circle of radius $r$.
For $r \in [2, \infty)$ angle $\phi(r)$ is minimized when $r=2$.
We compute $\phi \geq 17.8^{\circ}$, so $21$
discs will cover the annulus.


\end{proof}

\begin{proposition} \label{small covering sweepout}
Let $U \subset M$ be an open subset with 
boundary and suppose there exists $k$ sets $T_i \in \mathcal{T}$,
$k \leq 40$, such that $U \subset \bigcup T_i$.
Then there exists a monotone sweepout $z_t$ of $U$,
such that $l(z_t) \leq 489 \sqrt{\mu(U)}$.
\end{proposition}

We will inductively cut $U$ into smaller pieces
until the volume of each piece becomes so small that we can 
apply Proposition \ref{small volume sweepout}.
We will then use Proposition \ref{concat} to concatenate these sweepouts 
into one sweepout.

\begin{proposition} \label{small volume sweepout}
For every $\epsilon>0$ there exists a $\delta>0$, such that for every 
open set $U \subset M$ with $\mu(U) < \delta^2$ there exists a monotone sweepout 
$z_t$ of $U$ of length
$l(z_t)\leq  \epsilon$.
\end{proposition}

\begin{proof}
Choose $\delta >0$ be smaller then the injectivity radius and
suppose that it is small enough so that for every $x \in M$
and every $r\leq \delta$
the ball $B(x,r)$ with metric $g$ restricted to it
is $1.01$-bilipschitz diffeomorphic to a disc of radius $r$
in $\mathbb{R}^2$.

We will show that there exists a monotone sweepout 
of $U$ by cycles of length 
$\leq C \log(1 / \delta^2 ) \delta$,
where $C$ is a constant that does not depend on $\delta$
(but depends on the volume of $M$).
Note that we can make this quantity arbitrarily small 
by choosing sufficiently small $\delta$.

Choose a maximal collection of disjoint balls in $U$ of
radius $\delta/6$. Let $\mathcal{B}$ denote the collection of balls
with the same centers and radius $\delta/2$.
Observe that balls in $\mathcal{B}$ cover $U$.
Let $k$ denote the number of balls in $\mathcal{B}$.

We claim that there exists a monotone sweepout $z_t$
satisfying

\begin{equation} \label{log(k)}
l(z_t) \leq 500 log(k+1) \delta
\end{equation}

We prove equation (\ref{log(k)}) by induction on $k$.
Suppose $k \leq 100$. 
By coarea inequality for each $B_i \in \mathcal{B}$
there exists a concentric ball $B_i' \supset B_i$ of radius $r$,
$\delta/2 \leq r \leq \delta$, such that
$l(\partial B_i' \cap U) \leq 2 \delta$.
By Theorem \ref{Guth}
there exists a monotone sweepout of 
$U \cap B_i'$ by cycles of length at most $4 \delta$.
Let $B_j$ be a different ball in $\mathcal{B}$.
As for $B_i$ we can find a sweepout of $B_j' \cap (U \setminus B_i')$
for some $B_j' \supset B_j$, such that
$B_j'$ has radius $\leq \delta$ and 
$l(\partial B_j' \cap (U \setminus B_i'))\leq 2 \delta$.
By Lemma \ref{concat} there exists a monotone
sweepout of $(B_i' \cup B_j) \cap U$ by cycles of length
$\leq 6 \delta$. By repeating this step at most $100$
times we obtain a monotone sweepout of $U$ by cycles of length
at most $204 \delta$.

Assume the assertion holds for all $U$ that can be covered by 
$< k$ balls of radius $\frac{1}{2} \delta$.
Let $k'$ be the smallest integer greater or equal to
$k/100$ and 
let $B$ denote the union of $k'$ balls in $\mathcal{B}$.
By coarea inequality there exists $r \leq \delta/2$, s.t.
the boundary of the tubular neighbourhood
$\partial (N_{r}(B) \cap U))$ has length at most $2 \delta$. Set $U_1 = N_{r}(B) \cap U$.
Since $N_{r}(B)$ is contained in the $\delta/2$ neighbourhood of $B$, 
it can be covered by at most $k/10+1$ balls of radius $\delta/2$.
The set $U_2=U \setminus N_{r}(B) \cap U$ can be covered 
by $\frac{99}{100}k$ balls in $\mathcal{B}$.
By inductive assumption there exists 
a monotone sweepout of $U_i$, $i=1,2$, by cycles of length
$\leq 500 \log(\frac{99}{100}k+1) \delta$.
By Lemma \ref{concat} there exists a sweepout of
$U$ by cycles of length at most
$500 \log(\frac{99}{100}k) \delta + 2 \delta
< 500 \log(k) \delta $.
This completes the proof of equation \ref{log(k)}.

By definition of $\mathcal{B}$,
balls with the same centers and $1/3$ 
of the radius are disjoint. In particular,
the sum of their volumes is bounded above by $\mu(M)$.
It follows that $k\leq 12 \frac{\mu(M)}{\delta^2}$.
We conclude that 
$l(z_t) \leq C \log(1 / \delta^2 ) \delta$ as desired.
\end{proof}


We can now prove Proposition \ref{small covering sweepout}.
Let $\epsilon< 0.001 \sqrt{\mu(U)}$ 
be a small number
and choose $\delta(\epsilon)>0$ as in Lemma
\ref{small volume sweepout}.
We will prove that
for any subset $U' \subset U$ with piecewise smooth boundary 
there exists a monotone sweepout of 
$U'$ by cycles of length
$\leq 489 \sqrt{\mu(U')}$. 

The proof proceeds by induction on
$n = \log_{\frac{43}{44}}(\frac{\mu(U')}{\delta^2})$
and is reminiscent of arguments in \cite{LNR}.
When $\mu(U') \leq \delta^2$ we are done
by Lemma \ref{small volume sweepout}.
Assume the result to be true for all 
subsets of $\mu$- 
volume $\leq (\frac{44}{43})^{n-1} \delta^2$
and consider $U' \subset U$ with
$(\frac{44}{43})^{n-1}<\frac{\mu(U')}{\delta^2} \leq (\frac{44}{43})^{n}$.

Let $r$ be the smallest radius, such that
$\mu(B^0(x,r) \cap U') \geq \frac{\mu(U')}{44}$
for some $x \in M$.
By Lemma \ref{T} the intersection of the annulus
$B^0(x,3/2 r)\setminus B^0(x,r)$ with $U'$ can be covered by at most $42$ balls
$B^0(x_j,r)$. For each $j$ we have
$\mu (B^0(x_j,r) \cap U') \leq  \frac{\mu(U')}{44}$
since $B^0(x,r)$ has maximal $\mu$-volume for a ball of this radius.
It follows that the total $\mu$-volume of the set
$A=(B^0(x,3/2 r) \setminus B^0(x,r)) \cap U'$
is bounded by $\frac{42}{44} \mu(U')$. 
By Lemma \ref{length-area} we can find a relative
 cycle
$\gamma \subset A $ of length $\leq 2 \frac{\sqrt{\mu_0(U')}}{r} \sqrt{\mu(U')}$
separating $U'$ into two regions
each having $\mu$ volume less or equal to
$\frac{43}{44}\mu(U')$. Denote these two regions
by $U_1$ and $U_2$.

Now we derive a bound for the length of $\gamma$
that does not depend on $r$.
Since $U'$ can be covered by at most $40$ elements of $\mathcal{T}$
its $\mu_0$-volume is bounded by $40 \times 0.55
= 22$ (recall that $0.55$ is the maximal area of
an element in $\mathcal{T}$).
Hence, if $r>1.68$ we obtain that 
$l(\gamma) \leq 5.58 \sqrt{\mu(U')}$.

On the other hand, suppose $r \leq 1.68$.
In this case we can directly compute
(using a formula for the area of a disc in a space of 
constant curvautre $-1$, $0$ or $1$)
$\frac{2 \sqrt{\mu_0(A)}}{r} \leq 5.57$.

By inductive assumption both $U_1$ and $U_2$
admit a monotone sweepout with the desired length bound.
By Lemma \ref{concat} there exists a monotone 
sweepout of $U'$ by cycles of length
$\leq 489 \sqrt{\frac{43}{44}\mu(U')} +
5.58 \sqrt{\mu(U')} + \epsilon \leq
489 \sqrt{\mu(U')}$.

This concludes the proof of Proposition \ref{small covering sweepout}.

\section{Good covering of $M$} \label{sec: covering}

\begin{proposition} \label{nice covering}
Consider a surface $M$ and let $U \subset M$
be an open subset with piecewise smooth boundary
and suppose that it can be covered by $m$ elements of $\mathcal{T}$. 
Let $k$ be given.
Then there exists a collection $\mathcal{U}=\{U_i \}$ of at most $m+\max\{m,43 k\}$ 
sets, such that $\bigcup U_i$ covers $U$,
$\mu(U_i \cap U) \leq \frac{\mu(U)}{k}$,
each $U_i$ is contained in at most $40$
elements of $\mathcal{T}$ and
$l (int(U) \cap \bigcup \partial U_i)  \leq 
(94.6 \sqrt{m} + 36.6 \sqrt{\max\{m,43 k\}})
\sqrt{\mu(U)}$.
\end{proposition}

In the application of this Proposition
to the proof of Theorem \ref{main} we will take $U=M$.

\begin{proof}

Step 1. First we construct a covering of $U$
by sets $V_1,...,V_m$, such that each $V_i$ is contained
in at most $40$ polygons of $\mathcal{T}$,
and the union of their boundaries satisfies
a certain length bound.
The $\mu$-volume of each $V_i$, however, can be 
equal to anything between $0$ and $\mu(U)$.

Let $\mathcal{T}' \subset \mathcal{T}$ be the set of $m$
polygons that cover $U$ and let $T_l \in \mathcal{T}'$ be
such that $\mu(T_l\cap U) \geq \mu(T\cap U)$
for all $T \in \mathcal{T}'$.
By Proposition \ref{hyp covering} there are at most
$39$ polygons neighbouring $T_l$.
The intersection of each of them with $U$
has $\mu$-volume
less than or equal to $\mu(T_l\cap U)$.
By the length-area argument Lemma \ref{length-area}
we can find set $T'$ in the $1/2-$neighbourhood of $T_l$, $T_l\subset T' \subset N_{1/2}(T_l)$, such that
$l(\partial T' \cap U) \leq 2 \sqrt{39*0.55}  \sqrt{39} \sqrt{\mu(T_l\cap U)} < 58 \sqrt{\mu(T_l\cap U)}$.
We set $T'=V_1$. 
We now apply the same construction to select a set 
$V_2 \subset U \setminus V_1$, such that $V_2$ can be covered
by at most $40$ polygons in $\mathcal{T}'$
and $l(\partial V_2 \cap int(U \setminus V_1)) \leq 58 \sqrt{\mu(V_2)}$.
Each time we remove $V_i$ the number of polygons necessary 
to cover the remaining part of $U$ decreases by $1$.
Hence, we will be done after at most $m$ steps.
Since $V_i$ have disjoint interiors we have $\sum \mu(V_i) = \mu(U)$.
By Cauchy-Schwartz inequality the total length 
$l(int(U) \cap \bigcup \partial V_i)\leq 58 \sum \sqrt{\mu(V_i)}
\leq 58 \sqrt{m} \sqrt{\mu(U)}$.

Step 2.
Let $N= \max \{m,43k \}$.
We subdivide each of $V_i$ into a collection of
subsets $\mathcal{U}_i = \{U^i _j \}$, such that 
each $U^i_j$ has $\mu$-volume at most $\frac{43 \mu(U)}{N}$.
Let $k_i$ be the smallest integer larger than or equal to $N \mu(V_i)/\mu(U)$.
Observe that $\sum k_i \leq N+m$.

If $k_i = 1$ we set $\mathcal{U}_i = \{V_i \}$.
Suppose $k_i >1$. Let $B^0(x,r)$ be a ball with the property 
that $\mu(B^0(x,r) \cap V_i) = \frac{\mu(U)}{N}$
and $\mu(B^0(y,r) \cap V_i) \leq \mu(B^0(x,r) \cap V_i)$
for any $y \in M$. 
Since $V_i$ can be covered by at most $40$ polygons,
by Lemma \ref{T} we have that $B^0(x,3/2r)\cap V_i$ can be covered by
at most $43$ balls $B^0$ of radius $r$. 
It follows that $\mu$-volume of $B^0(x,3/2r)\cap V_i$
is at most $\frac{43 \mu(U)}{N} \leq \frac{\mu(U)}{k}$.

As in the proof of Proposition \ref{small covering sweepout}
we can bound $\mu_0$-volume of the annulus
$(B^0(x,3/2r)\setminus B^0(x,r)) \cap V_i$.
We separately consider the case when $r$ is small ($r\leq 1.68$),
and use comparison with the constant curvature space,
and the case when $r$ is large ($r> 1.68$) and
use upper bound on the area of $40$ polygons.
By Lemma \ref{length-area} we conclude that
there exists a set $U_1^i \supset B^0(x,r) \cap V_i$
of volume at most $\frac{43 \mu(U)}{N}$ and with
$l(int(V_i) \cap \partial U_1^i) \leq 5.58 \sqrt{\frac{43 \mu(U)}{N}}$.
Similarly, for each $j$ we can find subsets $U_j^i$ with disjoint interiors,
$\mu$-volume between $\frac{\mu(U)}{N}$ and $\frac{\mu(U)}{43N}$
and $l(\partial U_j^i \cap int(V_i \setminus (U_1^i \cup ... \cup U_{j-1}^i))) \leq 5.58 \sqrt{\frac{43 \mu(U)}{N}}$.
Observe that  $\mathcal{U}_i = \{U^i _j \}$ has at most $k_i$ elements.

We can now estimate the total length of the union of the boundaries
$L = l(int(U) \cap \bigcup_{i,j} \partial U_j ^i) 
\leq 58 \sqrt{m} \sqrt{\mu(U)}
+\sum k_i * 5.58 \sqrt{\frac{43 \mu(U)}{N}}$.
The second term is bounded by 
$36.6(\frac{m}{\sqrt{N}}+\sqrt{N})\sqrt{\mu(U)}$.
We conclude that the total length is bounded by
$(94.6 \sqrt{m} + 36.6 \sqrt{N})\sqrt{\mu(U)}$. 
\end{proof}

\section{Proof of Theorem \ref{main}}

Now we can prove Theorem \ref{main}.
Let $\mathcal{T}$ be a tessellation of $M$
by (at most) $\max\{1, 67g \}$ polygons as in Lemma \ref{T}.

By Proposition \ref{nice covering} we can cover $M$
by a collection of sets $U_i$ each of $\mu$-volume
at most $\mu(M)/k$ and contained in at most $40$
polygons of $\mathcal{T}$. The length
of the union of the boundaries of sets $U_i$
is bounded above 
by 
$(94.6 \sqrt{67 g}
   +36.6 \sqrt{\max \{67 g,43k \}}) \sqrt{\mu(M)}$.
Let $N$ denote the number
of sets in this covering.

First we construct a monotone 1-sweepout $z_t$ of $M$.
By Proposition \ref{small covering sweepout}
for each $U_i$ there exists a monotone sweepout of 
$U_i$ by cycles $z^i_t$ of length at most $489 \sqrt{\frac{\mu(M)}{k}}$.
For $j/N \leq t \leq (j+1)/N$ we set
$z_t = z^j_{Nt-j} + \sum_{i=1} ^{j-1} z^i_{1}$.
This defines a monotone sweepout of $M$
with the property that each cycle can be written as a sum
of chains
$z_t = c^1_t + c^2 _t$, where 
 $c^1_t$ has length at most $489 \sqrt{\frac{\mu(M)}{k}}$ and
 $c^2_t$ is contained in $\bigcup \partial U_i$.

Consider truncated symmetric product $TP^k(S^1)$,
i.e. all expressions of the form $\sum_{i=1} ^k a_i t_i$,
where $a_i \in \mathbb{Z}_2$ and $t_i \in S^1$.
For any $1$-sweepout $z_t$ the family of cycles 
$\{\sum_{i=1} ^k a_i z_{t_i}\}_{\sum_{i=1} ^k a_i t_i \in TP^p(S^1)}$
is a $k$-sweepout of $M$ (see \cite{Guth2008}, \cite{Glynn-AdeyLiokumovich}).

We estimate the mass of each cycle
$$   l(\sum_{i=1} ^k a_i z_{t_i})
\leq k \max_{t} \{l(c^1_t) \} + l(\bigcup \partial U_i) $$
$$ \leq (489 \sqrt{k} + 94.6 \sqrt{67 g}
   +36.6 \sqrt{\max \{67 g,43k \}}) \sqrt{\mu(M)}$$
   
In particular, $ l(\sum_{i=1} ^k a_i z_{t_i})
\leq 1600 \max \{\sqrt{k},\sqrt{g} \} \sqrt{\mu(M)}$.
This concludes the proof of Theorem \ref{main}.


\bibliographystyle{abbrv}
\bibliography{bibliography}

\end{document}